\documentclass[11pt,reqno]{amsart}
\usepackage{graphicx}
\usepackage{amsfonts}
\usepackage{mathrsfs}
 \usepackage{times}
 \usepackage{mathptmx}
\usepackage{amsmath}
\usepackage{txfonts}
\usepackage{amssymb}
\usepackage{amsmath,amsthm}
\usepackage{CJK,amsfonts}
\usepackage[body={7.5in,8.75in}, top=1.2in, right=0.9in,left=0.9in]{geometry}
\usepackage[colorlinks=true, linkcolor=red, citecolor=blue]{hyperref}
\allowdisplaybreaks[4] \numberwithin{equation}{section}

\newtheorem{thm}{Theorem}

\newtheorem{corollary}[thm]{Corollary}
\newtheorem{lemma}[thm]{Lemma}
\newtheorem{remark}[thm]{Remark}

\def\squarebox#1{\hbox to #1{\hfill\vbox to #1{\vfill}}}

\newcommand{\be}{\begin{equation}}
\newcommand{\ee}{\end{equation}}
\newcommand{\bea}{\begin{eqnarray}}
\newcommand{\eea}{\end{eqnarray}}
\newcommand{\bd}{\begin{displaymath}}
\newcommand{\ed}{\end{displaymath}}
\begin{document}
\begin{CJK*}{GBK}{song}
\title[Generalized $q$-difference equations for general $q$-polynomials]{Generalized $q$-difference equations for general $q$-polynomials with double $q$-binomial coefficients}
\author{ Jian Cao${}^{1}$, Sama Arjika${}^{2,3}$  and Mahouton Norbert Hounkonnou${}^{3}$}
\dedicatory{\textsc}
\thanks{${}^1$School of Mathematics, Hangzhou Normal University, Hangzhou City, Zhejiang Province, 311121, China.  ${}^2$Department of Mathematics and Informatics, University of Agadez, Niger. ${}^3$International Chair in Mathematical Physics and Applications (ICMPA-UNESCO Chair),
University of Abomey-Calavi,  072 BP 50,  , Cotonou, Benin Republic. }
\thanks{Email:21caojian@hznu.edu.cn, 21caojian@163.com, rjksama2008@gmail.com, norbert.hounkonnou@cipma.uac.bj.}



\begin{abstract}
In this  paper, we use the generalized $q$-polynomials with double $q$-binomial coefficients and homogeneous $q$-operators [J.  Difference  Equ. Appl. {\bf20} (2014), 837--851.] to construct $q$-difference equations with seven variables, which generalize recent works of Jia {\it et al} [Symmetry {\bf 2021}, 13, 1222.]. In addition, we derive Rogers formulas, extended Rogers formulas and Srivastava--Agarwal type bilinear generating functions  for generalized  $q$-polynomials, which generalize generating functions for Cigler's polynomials [J. Difference Equ. Appl. {\bf 24} (2018),   479--502.]. Finally,  we also derive  mixed generating functions using $q$-difference equations.
\end{abstract}

\maketitle

\section{\bf Introduction}
In this paper, we adopt the common   notation and terminology for basic hypergeometric  series as in Refs.  \cite{GasparRahman,Koekock}.    Throughout this paper, we assume that  $q$ is a fixed nonzero real or complex number and $|q|< 1$.  The $q$-shifted factorial  and its compact  factorial  are defined \cite{GasparRahman,Koekock}, respectively by:
\begin{equation}
(a;q)_0:=1,\quad  (a;q)_{n} \textcolor{red}{:}=\prod_{k=0}^{n-1} (1-aq^k),   \; (a;q)_{\infty}:=\prod_{k=0}^{\infty}(1-aq^{k})
\end{equation}
and
 $ (a_1,a_2, \ldots, a_r;q)_n=(a_1;q)_n (a_2;q)_n\cdots(a_r;q)_n,\; n\in\{0, 1, 2\cdots\}$. \\
 We will use frequently  the following relation
 \begin{equation}
 \label{usu}
(aq^{-n};q)_n=(q/a;q)_n(-a)^nq^{-n-({}^n_2)}.
\end{equation}
  The generalized $q$-binomial coefficient is defined as \cite{GasparRahman}
\be
\label{qb}
 {\,\alpha\,\atopwithdelims []\,k\,}_{q}=\frac{(q^{-\alpha};q)_k}{(q;q)_k}(-1)^kq^{\alpha k-({}^k_2)}, \quad\,\alpha \in\mathbb{C}.
\ee
Similarly, by replacing $q$ by $-q$ in the above relation, we get the following:
\be
\label{qb1}
  {\,\alpha\,\atopwithdelims []\,k\,}_{-q}=\frac{(-q^{-\alpha};q)_k}{(-q;q)_k} q^{\alpha k-({}^k_2)},\quad\,\alpha \in\mathbb{C}.
\ee
\par
The   basic or $q$-hypergeometric function    in the variable $z$ (see  Slater \cite[Chap. 3]{SLATER},  Srivastava and Karlsson   \cite[p. 347, Eq. (272)]{SrivastaKarlsson}   for details) is defined as:
 $$
{}_{r}\Phi_s\left[\begin{array}{r}a_1, a_2,\ldots, a_r;
 \\\\
b_1,b_2,\ldots,b_s;
 \end{array}
q;z\right]
 =\sum_{n=0}^\infty\Big[(-1)^n q^{({}^n_2)}\Big]^{1+s-r}\,\frac{(a_1, a_2,\ldots, a_r;q)_n}{(b_1,b_2,\ldots,b_s;q)_n}\frac{ z^n}{(q;q)_n},
$$
 when $r>s+1$. Note that, for $r=s+1$, we have:
$$
{}_{r+1}\Phi_r\left[\begin{array}{r}a_1, a_2,\ldots, a_{r+1};
 \\\\
b_1,b_2,\ldots,b_r;
 \end{array} q; z\right]
 =\sum_{n=0}^\infty \frac{(a_1, a_2,\ldots, a_{r+1};q)_n}{(b_1,b_2,\ldots,b_r;q)_n}\frac{ z^n}{(q;q)_n}.
$$
 \par
We remark in passing that, in a recently-published
survey-cum-expository review article, the so-called $(p,q)$-calculus
was exposed to be a rather trivial and inconsequential variation of
the classical $q$-calculus, the additional parameter $p$ being redundant
or superfluous (see, for details, \cite[p. 340]{HMS-ISTT2020}).\par
  Chen {\it et al.} \cite{Chen2003} introduced  homogeneous $q$-difference operator   $D_{xy}$
  as
\be
\label{deffd}
  D_{xy}\big\{f(x,y)\}:=\frac{f(x,q^{-1}y)-f( qx, y)}{x-q^{-1}y},
\ee
which  turn out to be suitable for dealing with the Cauchy polynomials.

Wang and Cao \cite{CAO2018} presented two extensions of Cigler's (see \cite{SA2021})
polynomials
\be
\label{def02}
\mathcal{C}_n^{(\alpha-n)}(x,y,b)=
\sum_{k=0}^n (-1)^kq^{({}^k_2)} { \alpha\,\atopwithdelims []\,k\,}_{q}b^k \frac{(q;q)_n}{(q;q)_{n-k}}   p_{n-k}(x,y)
\ee
and
\be
\label{def03}
\mathcal{D}_n^{(\alpha-n)}(x,y,b)=
\sum_{k=0}^nq^{({}^k_2)}  { \alpha\,\atopwithdelims []\, k\,}_{q}b^k\frac{(q;q)_n}{(q;q)_{n-k}}\left[(-1)^{n+k} q^{-({}^n_2)}  p_{n-k}(y,x)\right],
\ee
where $\displaystyle p_n(x,y):=(x-y)( x- qy)\cdots ( x-q^{n-1}y) =(y/x;q)_n\,x^n$ are  the Cauchy polynomials.

Recently, Jia {\it et al} \cite{ZJia-BKhan2021}  introduced the following
\be
\label{dref03}
L_{\tilde{m},\tilde{n}} (\alpha,x,z,a)=
\sum_{k=0}^n   { n\,\atopwithdelims []\, k\,}_{q} { \alpha\,\atopwithdelims []\, k\,}_{-q}   q^{\tau(\tilde{m},\tilde{n})+({}^k_2)} (a;q)_k z^kx^{n-k},
\ee
with
\be
\tau(\tilde{m},\tilde{n})=\tilde{m}({}^k_2)-\tilde{n}({}^{k+1}_{\;\,\;2}),\label{tau}
\ee
 where  $\tilde{m}$ and $\tilde{n}$ are real numbers.
\par

Our present investigation is essentially motivated  by
the earlier works by Jia {\it et al} \cite{ZJia-BKhan2021}. Our aim here is to introduce and study some further extensions
of the above-mentioned $q$-polynomials
\be
\label{dedref03}
\tilde{L}_n^{(\tilde{r},\tilde{s})} (\alpha,x,y,z,a)=
\sum_{k=0}^n   { n\,\atopwithdelims []\, k\,}_{q} { \alpha\,\atopwithdelims []\, k\,}_{-q}   q^{\tau(\tilde{r},\tilde{s})+({}^k_2)} (a;q)_k p_{n-k}(x,y)z^k,
\ee
where $\tau(\tilde{r},\tilde{s})$ is defined as in (\ref{tau}).
 
The main task in this paper is to show how  the $q$-polynomials $(\ref{dedref03})$ are related to  other known polynomials.  For example, the generalized $q$-polynomials
$\tilde{L}_n^{(\tilde{r},\tilde{s})} (\alpha,x,y,z,a)$ defined
in $(\ref{dedref03})$  are   generalized  and unified forms of
 Hahn polynomials and  Al-Salam-Carlitz polynomials. Particular cases of known results in the literature are given in Remark \ref{remark1}.  
\begin{remark}
\label{remark1}
$\;$
\begin{enumerate}
\item Upon setting $y=0$,
the generalized $q$-polynomials
$\tilde{L}_n^{(\tilde{r},\tilde{s})} (\alpha,x,y,z,a)$ defined in  $(\ref{dedref03})$
 reduce to (\ref{dref03})  (see \cite{ZJia-BKhan2021}):
\be
\tilde{L}_n^{(\tilde{r},\tilde{s})} (\alpha,x,0,z,a)
= L_{\tilde{r},\tilde{s}} (\alpha,x,z,a).
\ee
\item For $(\alpha,\tilde{r},\tilde{s},x,y,z,a)=(\infty,0,0, y,x,-z,-q)$, the generalized
$q$-polynomials $\tilde{L}_n^{(\tilde{r},\tilde{s})} (\alpha,x,y,z,a)$  reduce to trivariate $q$-polynomials $F_n(x,y,z;q)$ \cite{MAA2016}
\be
 \tilde{L}_n^{(0,0)} (\infty,y,x,-z,-q)
=(-1)^n q^{({}^n_2)}F_n(x,y,z;q).
\ee
\item Upon setting $\alpha=n\in\mathbb{Z}$ and $(\tilde{r},\tilde{s}, a, x,y,z)=(0, -1, -yq, 1,0,x)$,
the  $q$-polynomials
$\tilde{L}_n^{(\tilde{r},\tilde{s})} (\alpha,x,y,z,a)$ 
  reduce to    $\rho_e(n,y,x,q)$ (see \cite{ZJia-BKhan2021}):
\be
\tilde{L}_n^{(0,-1)} (n,1,0,x,-qy)
=\rho_e(n,y,x,q).
\ee
\item  Upon setting  $(\alpha,\tilde{r},\tilde{s},y,a)=(\infty,-1,0,  1, -q)$, the generalized
$q$-polynomials $\tilde{L}_n^{(\tilde{r},\tilde{s})} (\alpha,x,y,z,a)$  reduce to the  homogeneous Rogers-Szeg\"o polynomials
$h_n(x,y|q)$ (see \cite{LHS-AAS2013}):
\be
\tilde{L}_n^{(-1,0)} (\infty,x,y,1,-q)
=h_n(x,y|q).
\ee
\item By choosing  $(\alpha,\tilde{r},\tilde{s}, a, x,y)=(\infty,-1, 0, -q, xq^{-n},0)$, the generalized
$q$-polynomials $\tilde{L}_n^{(\tilde{r},\tilde{s})} (\alpha,x,y,z,a)$  reduce to the  Rogers-Szeg\"o polynomials
$g_n(z,x|q)$ (see \cite{AlSalam}):

\be
\tilde{L}_n^{(-1,0)} (\infty,xq^{-n},0,z,-q)
=g_n(z,x|q).
\ee
\end{enumerate}
\end{remark}

 The rest of paper is organized as follows. In section \ref{section2}, we deduce  the main results of $q$-difference equations with seven-variable for generalized $q$-polynomials. In section \ref{section3}, we obtain the generating function of generalized $q$-polynomials by the method of $q$-difference equations. In section \ref{section4}, we derive Rogers formula for generalized $q$-polynomials by using the $q$-difference equations. In section \ref{section5}, we gain a mixed generating function for generalized $q$-polynomials by $q$-difference equations.

\section{Main results}\label{section2}
In this section, we give the following  fundamental  theorem.
\begin{thm}\label{thm1}
Let $f(\alpha,x,y,a,z,\tilde{r},\tilde{s})$ be a seven-variable analytic function in a neighborhood of $(\alpha,x,y,a,z,\tilde{r},\tilde{s})=(0,0,0,0,0,0,0)\in\mathbb{C}^7.$ Then, $f(\alpha,x,y,a,z,\tilde{r},\tilde{s})$ can be expanded in terms of  $\tilde{L}_n^{(\tilde{r},\tilde{s})} (\alpha,x,y,z,a)$  if and only if $f$ satisfies the following $q$-difference equation:
\begin{align}\label{thm1_1}
 & (x-q^{-1}y)\Big\{f(\alpha,x,y,a,z,\tilde{r},\tilde{s})-f(\alpha,x,y,a,q^2z,\tilde{r},\tilde{s})\Big\} \nonumber\\
 &\quad\qquad\qquad\qquad\qquad =q^{\alpha-\tilde{s}}z\Big\{f(\alpha,qx,y,a,zq^{ \tilde{r}-\tilde{s}},\tilde{r},\tilde{s})-f(\alpha,x,q^{-1}y,a,zq^{ \tilde{r}-\tilde{s}},\tilde{r},\tilde{s})\Big\}\nonumber \\
&\quad\qquad\qquad\qquad\qquad  +q^{-\tilde{s}}(1-aq^{\alpha })\Big\{f(\alpha,qx,y,a,zq^{1+\tilde{r}-\tilde{s}},\tilde{r},\tilde{s})-f(\alpha,x,yq^{-1},a,zq^{1+\tilde{r}-\tilde{s}},\tilde{r},\tilde{s})\Big\} \nonumber\\
&\quad\qquad\qquad\qquad\qquad   -azq^{-\tilde{s}}\Big\{f(\alpha,qx,y,a,zq^{2+\tilde{r}-\tilde{s}},\tilde{r},\tilde{s})-f(\alpha,x,yq^{-1},a,zq^{2+\tilde{r}-\tilde{s}},\tilde{r},\tilde{s})\Big\}.
\end{align}
\end{thm}

\begin{remark}
For $y=0$ in Theorem \ref{thm1}, we get the concluding remarks of Jia {\it et al} \cite{ZJia-BKhan2021}.
\end{remark}

To determine whether a given function is an analytic function in several complex variables, we often use the following Hartogs's theorem. For more information, please refer to \cite{19,7}.

\begin{lemma}[{{\cite[Hartogs's theorem]{9}}}]
If a complex-valued function is separately holomorphic (analytic)  in each variable  in an open domain $D \in\mathbb{C}^n$, then it is holomorphic (analytic) in $D$.
\end{lemma}

In order to prove Theorem \ref{thm1}, we need the following fundamental property of several complex variables.
\begin{lemma}[{\cite[Proposition 1]{3}}]
If $f(x_1,x_2,...,x_k)$ is analytic at the origin $(0,0,...,0)\in\mathbb{C}^k$, then, $f$ can be expanded in an absolutely convergent power series,
\begin{align*}
f(x_1,x_2,...,x_k)=\sum_{n_1,n_2,...,n_k=0}^\infty\alpha_{n_1,n_2,...,n_k}x_1^{n_1}x_2^{n_2}...x_k^{n_k}.
\end{align*}
\end{lemma}

\begin{proof}[Proof of Theorem \ref{thm1}]
  From the Hartogs's theorem and the theory of several complex variables, we assume that
\begin{align}\label{27}
f(\alpha,x,y,a,z,\tilde{r},\tilde{s})=\sum_{k=0}^\infty A_k(\alpha,x,y,a,\tilde{r},\tilde{s})z^k.
\end{align}\par
On one hand, substituting \eqref{27} into \eqref{thm1_1} yields
\begin{multline}
 (x-q^{-1}y)\sum_{k=0}^\infty(1-q^{2k}) A_k(\alpha,x,y,a,\tilde{r},\tilde{s})z^k\\
=\sum_{k=0}^\infty\Big\{q^{\alpha-\tilde{s}+k(\tilde{r}-\tilde{s})}+q^{-\tilde{s}+k(1+\tilde{r}-\tilde{s})}(1-aq^{\alpha}) -aq^{-\tilde{s}+k(2+\tilde{r}-\tilde{s})}\Big\}\Big\{A_k(\alpha,x,q^{-1}y,a,\tilde{r},\tilde{s})- A_k(\alpha,qx,y,a,\tilde{r},\tilde{s})\Big\}z^{k+1},
\end{multline}
which is equal to
\begin{multline}\label{A-n rel-eq}
 (x-q^{-1}y)\sum_{k=0}^\infty(1-q^{2k}) A_k(\alpha,x,y,a,\tilde{r},\tilde{s})z^k\\
=\sum_{k=0}^\infty q^{-\tilde{s}+k(\tilde{r}-\tilde{s})}\Big\{q^\alpha+q^{ k }(1-aq^{\alpha }) -aq^{2 k }\Big\}\Big\{A_k(\alpha,x,q^{-1}y,a,\tilde{r},\tilde{s})- A_k(\alpha,qx,y,a,\tilde{r},\tilde{s})\Big\}z^{k+1}.
\end{multline}\par
Equating coefficients of $z^k,\,k\geq 1$ on both sides of equation \eqref{A-n rel-eq}, we see that
\begin{multline}
 (x-q^{-1}y)(1-q^k)(1+q^{k }) A_k(\alpha,x,y,a,\tilde{r},\tilde{s}) \\
= q^{-\tilde{s}+(k-1)(\tilde{r}-\tilde{s})}(q^\alpha+q^{ k-1 })(1-aq^{k-1})\Big[A_{k-1}(\alpha,x,q^{-1}y,a,\tilde{r},\tilde{s})-A_{k-1}(\alpha,qx,y,a,\tilde{r},\tilde{s})\Big],
\end{multline}
or
\begin{align*}
A_k(\alpha,x,y,a,\tilde{r},\tilde{s})&=q^{\tilde{r}(k-1)-\tilde{s}k }\frac{(q^\alpha+q^{ k-1 })(1-aq^{k-1})}{(1-q^k)(1+q^{k }) }\cdot\frac{A_{k-1}(\alpha,x,q^{-1}y,a,\tilde{r},\tilde{s})-A_{k-1}(\alpha,qx,y,a,\tilde{r},\tilde{s})}{x-q^{-1}y} \\
&=q^{\alpha+\tilde{r}(k-1)-\tilde{s}k }\frac{(1+q^{-\alpha+ k-1 })(1-aq^{k-1})}{(1-q^k)(1+q^{k }) }\cdot D_{xy}\{A_{k-1}(\alpha,x,y,a,\tilde{r},\tilde{s})\}.
\end{align*}
The iteration then reveals that:
\begin{align*}
A_k(\alpha,x,y,a,\tilde{r},\tilde{s})=q^{k\alpha+\tilde{r}({}^k_2)-\tilde{s}({}^{k+1}_{\;\,\;2})}\frac{(-q^{-\alpha},a;q)_k }{(q^2;q^2)_k }\cdot D_{xy}^k\{A_0(\alpha,x,y,a,\tilde{r},\tilde{s})\}.
\end{align*}\par
Letting $\displaystyle f(\alpha,x,y,a,0,\tilde{r},\tilde{s})=A_0(\alpha,x,y,a,\tilde{r},\tilde{s})=\sum_{n=0}^\infty\mu_np_n(x,y)$ yields
\begin{align}\label{28}
A_k(\alpha,x,y,a,\tilde{r},\tilde{s})=q^{k\alpha+\tilde{r}({}^k_2)-\tilde{s}({}^{k+1}_{\;\,\;2})}\frac{(-q^{-\alpha},a;q)_k}{(q^2;q^2)_k }\cdot\sum_{n=0}^\infty\mu_n\frac{(q;q)_n}{(q;q)_{n-k}}p_{n-k}(x,y),
\end{align}
and we have
\begin{align*}
 f(\alpha,x,y,a,z,\tilde{r},\tilde{s})&=\sum_{k=0}^\infty q^{k\alpha+\tilde{r}({}^k_2)-\tilde{s}({}^{k+1}_{\;\,\;2})}\frac{(-q^{-\alpha},a;q)_k }{(q^2;q^2)_k }\sum_{n=0}^\infty\mu_n\frac{(q;q)_n}{(q;q)_{n-k}}p_{n-k}(x,y)z^k\\
&=\sum_{n=0}^\infty\mu_n\sum_{k=0}^n\begin{bmatrix}
                        n \\
                       k \\
    \end{bmatrix}_q\begin{bmatrix}
                        \alpha \\
                       k \\
    \end{bmatrix}_{-q}q^{\tau(\tilde{r},\tilde{s})+({}^k_2)} (a;q)_kp_{n-k}(x,y)z^k\\
&=\sum_{n=0}^\infty\mu_n \tilde{L}_n^{(\tilde{r},\tilde{s})} (\alpha,x,y,z,a).
\end{align*}
\par
On the other hand, if $f(\alpha,x,y,a,z,\tilde{r},\tilde{s})$ can be expanded in terms of 
$\tilde{L}_n^{(\tilde{r},\tilde{s})} (\alpha,x,y,z,a)$, we verify that $f(\alpha,x,y,a,z,\tilde{r},\tilde{s})$ satisfies equation \eqref{thm1_1}.  The proof of Theorem \ref{thm1} is complete.
\end{proof}

\section{Generating function of generalized $q$-polynomials}\label{section3}
In this section, we give the generating function of generalized $q$-polynomials by the method of $q$-difference equations.

\begin{thm}\label{thmc3}
Let $ \tilde{L}_n^{(\tilde{r},\tilde{s})} (\alpha,x,y,z,a)$ be the polynomials defined as in (\ref{dedref03}).
\begin{itemize}
\item For  $|xt|<1,$ and  $\tilde{r},\,\tilde{s} \in\mathbb{Z}^\ast$, we have:
\begin{align}\label{thmc3_1}
\sum_{n=0}^\infty \tilde{L}_n^{(\tilde{r},\tilde{s})} (\alpha,x,y,z,a)\frac{t^n}{(q;q)_n}=\frac{(yt;q)_\infty}{(xt;q)_\infty} \sum_{k=0}^\infty  \frac{(-q^{-\alpha},a;q)_k }{(q^2;q^2)_k } q^{k\alpha+\tilde{r}({}^k_2)-\tilde{s}({}^{k+1}_{\;\,\;2})} (zt)^k.
\end{align}
\item  For   $\tilde{r}=\tilde{s}=0$ in (\ref{thmc3_1}) and $\max\{|zt|,|xt|\}<1,$ we recover the Cauchy polynomials as follows:
\begin{align}\label{thmc3_2}
\sum_{n=0}^\infty \frac{p_{n}(x,y)\,t^n}{(q;q)_n}\sum_{k=0}^n   { n\,\atopwithdelims []\, k\,}_{q} { \alpha\,\atopwithdelims []\, k\,}_{-q}  \frac{  (-1)^kq^{k(k-n) }(a;q)_k}{  p_{k}(y, xq^{1-n})}z^k
=\frac{(yt;q)_\infty}{(xt;q)_\infty}  {}_{2}\Phi_1\left[\begin{array}{r}-q^{-\alpha}, a;
 \\\\
-q;
 \end{array}
q; ztq^\alpha \right].
\end{align}
\end{itemize}
\end{thm}

\begin{proof}[Proof of Theorem \ref{thmc3}]
By the Weierstrass $M$-test, the series $\displaystyle \sum_{n=0}^\infty M_n$ is convergent when
$\displaystyle \lim_{n\to +\infty}\left|\frac{M_{n+1}}{M_n}\right|<1$. 
We check that both sides of (\ref{thmc3_1}) are convergent if $|xt|<1,$ that is
\begin{align*}
&\lim_{n\to\infty}\left|
\frac{\tilde{L}_{n+1}^{(\tilde{r},\tilde{s})} (\alpha,x,y,z,a) \;t^{n+1}/ (q;q)_{n+1}}{\tilde{L}_n^{(\tilde{r},\tilde{s})} (\alpha,x,y,z,a)\; t^n/ (q;q)_{n}} 
\right|=|xt|<1,\\
&
\lim_{n\to\infty}\left|
\frac{  t^{n+1} \sum_{k=0}^n\begin{bmatrix}
                        n+1 \\
                       k \\
    \end{bmatrix}_q\begin{bmatrix}
                        \alpha \\
                       k \\
    \end{bmatrix}_{-q}q^{\tau(\tilde{r},\tilde{s})+({}^k_2)} (a;q)_kp_{n+1-k}(x,y)z^k/ (q;q)_{n+1}}{ t^{n} \sum_{k=0}^n\begin{bmatrix}
                        n\\
                       k \\
    \end{bmatrix}_q\begin{bmatrix}
                        \alpha \\
                       k \\
    \end{bmatrix}_{-q}q^{\tau(\tilde{r},\tilde{s})+({}^k_2)} (a;q)_kp_{n-k}(x,y)z^k/ (q;q)_{n}} 
\right|=|xt|<1.
\end{align*}
Denoting  by $f(\alpha,x,y,a,z,\tilde{r},\tilde{s})$  the right-hand side of equation \eqref{thmc3_1}, it  can equivalently be written  by
\begin{align}\label{44x}
f(\alpha,x,y,a,z,\tilde{r},\tilde{s})&= \sum_{k=0}^\infty  \frac{(-q^{-\alpha},a;q)_k }{(q^2;q^2)_k } q^{k\alpha+\tilde{r}({}^k_2)-\tilde{s}({}^{k+1}_{\;\,\;2})} z^k\cdot \frac{t^k\;(yt;q)_\infty}{(xt;q)_\infty}\cr
&=  \sum_{k=0}^\infty  \frac{(-q^{-\alpha},a;q)_k }{(q^2;q^2)_k } q^{k\alpha+\tilde{r}({}^k_2)-\tilde{s}({}^{k+1}_{\;\,\;2})} z^k D_{xy}^k\left\{ \frac{(yt;q)_\infty}{(xt;q)_\infty}\right\}.
\end{align}
Letting $\displaystyle
f(\alpha,x,y,a,z,\tilde{r},\tilde{s})=\sum_{k=0}^\infty A_k(\alpha,x,y,a,\tilde{r},\tilde{s})z^k$ and
\begin{align}
A_k(\alpha,x,y,a,\tilde{r},\tilde{s})=q^{k\alpha+\tilde{r}({}^k_2)-\tilde{s}({}^{k+1}_{\;\,\;2})}\frac{(-q^{-\alpha},a;q)_k}{(q^2;q^2)_k }  D_{xy}^k\left\{ \frac{(yt;q)_\infty}{(xt;q)_\infty}\right\}, \label{qt}
\end{align}
we obtain
\begin{align}
A_0(\alpha,x,y,a,\tilde{r},\tilde{s})=  \frac{(yt;q)_\infty}{(xt;q)_\infty},\label{dcqt}
\end{align}
and $f(\alpha,x,y,a,0,\tilde{r},\tilde{s})=  A_0(\alpha,x,y,a,\tilde{r},\tilde{s})$. Taking (\ref{dcqt}) into (\ref{qt}), we get:
 \begin{align}
A_k(\alpha,x,y,a, \tilde{r},\tilde{s})=q^{k\alpha+\tilde{r}({}^k_2)-\tilde{s}({}^{k+1}_{\;\,\;2})}\frac{(-q^{-\alpha},a;q)_k}{(q^2;q^2)_k }\cdot D_{xy}^k\{A_0(\alpha,x,y,a, \tilde{r}, \tilde{s})\}.
\end{align}
Through above identities, $f(\alpha,x,y,a,z,\tilde{r},\tilde{s})$ satisfies the equation (\ref{thm1_1}). So, we have
\begin{align*}
f(\alpha,x,y,a,z,\tilde{r},\tilde{s})&=
  \sum_{k=0}^n\frac{(-q^{-\alpha},a;q)_k }{(q^2;q^2)_k } q^{k\alpha+\tilde{r}({}^k_2)-\tilde{s}({}^{k+1}_{\;\,\;2})} z^k \sum_{n=0}^\infty\frac{p_{n}(x,y)\;t^{n+k}}{(q;q)_{n }}
\cr&=  \sum_{n=0}^\infty  \sum_{k=0}^n\frac{(-q^{-\alpha},a;q)_k }{(q^2;q^2)_k } q^{k\alpha+\tilde{r}({}^k_2)-\tilde{s}({}^{k+1}_{\;\,\;2})} z^k\frac{p_{n-k}(x,y)\;t^n}{(q;q)_{n-k}}\cr
&=\sum_{n=0}^\infty\frac{t^n}{(q;q)_n}\sum_{k=0}^n\begin{bmatrix}
                        n \\
                       k \\
    \end{bmatrix}_q\begin{bmatrix}
                        \alpha \\
                       k \\
    \end{bmatrix}_{-q}q^{\tau(\tilde{r},\tilde{s})+({}^k_2)} (a;q)_kp_{n-k}(x,y)z^k\cr
&=\sum_{n=0}^\infty\frac{t^n}{(q;q)_n}\tilde{L}_n^{(\tilde{r},\tilde{s})} (\alpha,x,y,z,a),
\end{align*}
which  is the left-hand side of (\ref{thmc3_1}).  The proof is complete.
\end{proof}
\begin{remark}
Setting $y=0$, in (\ref{thmc3_1}), we get the concluding remarks of \cite{ZJia-BKhan2021}:
\begin{align}\label{thmccc3_1}
\sum_{n=0}^\infty L_{\tilde{r},\tilde{s}} (\alpha,x,z,a)\frac{t^n}{(q;q)_n}=\frac{1}{(xt;q)_\infty} \sum_{k=0}^\infty  \frac{(-q^{-\alpha},a;q)_k }{(q^2;q^2)_k } q^{k\alpha+\tilde{r}({}^k_2)-\tilde{s}({}^{k+1}_{\;\,\;2})} (zt)^k,\; \;\; |xt|<1.
\end{align}
For
$\alpha \to \infty,\,\tilde{r}=\tilde{s}=0,\,x\to y,\,y\to x,\,z=-z$ and $a=-q$, in (\ref{thmc3_1}), we get the concluding remarks of   \cite{MAA2016}:
\begin{align}\label{ctr}
\sum_{n=0}^\infty F_n(x,y,z;q)\frac{(-1)^n q^{({}^n_2)}\,t^n}{(q;q)_n}=\frac{(xt, zt;q)_\infty}{(yt;q)_\infty},\; \;\; |yt|<1.
\end{align}
\end{remark} 
\section{Rogers formula for generalized $q$-polynomials}\label{section4}
In this section, we give and prove the Rogers formula for generalized $q$-polynomials by using the $q$-difference equations, so that we derive  Rogers formula for the trivariate $q$-polynomials:
\par
Chen and Liu \cite{WYCC-ZGL1997} have defined the $q$-exponential operator as follows  (see \cite{LHS-AAS2013}):
\begin{align}
T(bD_a)=\sum_{n=0}^\infty \frac{(bD_a)^n}{(q;q)_n},
\end{align}
where the usual $q$-differential operator, or $q$-derivative, is defined by
\be
D_af(a)=\frac{f(a)-f(qa)}{a}.
\ee
The Leibniz rule for $D_a$ is the following identity, which is a variation of $q$-binomial theorem \cite{Roman1985}:
\begin{align}
\label{RUe}
    D_a^n\{f(a)g(a)\} =\sum_{k=0}^nq^{k(k-n)}\begin{bmatrix}n \\k \\ \end{bmatrix}_qD_a^k\{f(a)\}D_a^{n-k}\Bigl\{g\bigl(aq^k\bigr)\Bigr\},
\end{align}
where $D_q^0$ is understood as the identity.
The following property of $D_a$ is straightforward, but important.
\begin{lemma} 
\label{dAMM}
\begin{equation}
D_a^n\left\{\frac{(as;q)_\infty}{(a\omega;q)_\infty}\right\}
= \omega^n\; \frac{(s/\omega;q)_n}
{(as;q)_n}\;\frac{(as;q)_\infty}{( a\omega;q)_\infty}.\label{aberll}
\end{equation}

\end{lemma}

\begin{lemma}
\label{dsdpqro1}

For $k\in\mathbb{N}_0$ and $\max\{|x\omega|, |xt|\}<1$, we have:
\begin{align}
\label{dte}
T(tD_\omega)\left\{\frac{(y\omega;q)_\infty}{(x\omega;q)_\infty}  \omega^k\right\}=\frac{ (y\omega;q)_\infty}{(x\omega;q)_\infty}  \omega^k \sum_{j=0}^k  \frac{(-1)^jq^{kj-({}^j_2)} (q^{-k},x\omega;q)_j\, (t/\omega)^j }{(y\omega ,q;q)_j}    {}_{2}\Phi_1\left[\begin{array}{r}y/x,0;
 \\\\
y\omega q^j;
 \end{array}
q; xt\right].
\end{align}
\end{lemma}

\begin{proof}
By    means of the  Leibniz rule (\ref{RUe}),  the left-hand side of (\ref{dte}) equals
\begin{align*}
 & \sum_{n=0}^\infty\frac{t^n}{(q;q)_n} \sum_{j=0}^nq^{j(j-n)}\begin{bmatrix}n \\j \\ \end{bmatrix}_qD_\omega^{j} \{\omega^k\}D_\omega^{n-j}\left\{\frac{(y\omega q^j;q)_\infty}{(x\omega q^j;q)_\infty} \right\}\cr
& \qquad \qquad\qquad= \sum_{n=0}^\infty\sum_{j=0}^n\frac{q^{j(j-n)}\;t^n}{(q;q)_j(q;q)_{n-j}} \frac{(q;q)_k}{(q;q)_{k-j}} \omega^{k-j} D_\omega^{n-j}\left\{\frac{(y\omega q^j;q)_\infty}{(x\omega q^j;q)_\infty} \right\}\cr
& \qquad \qquad\qquad=  \sum_{n=0}^\infty\sum_{j=0}^\infty\frac{q^{-jn}\;t^{n+j}}{(q;q)_j(q;q)_{n}} \frac{(q;q)_k}{(q;q)_{k-j}} \omega^{k-j} D_\omega^{n}\left\{\frac{(y\omega q^j;q)_\infty}{(x\omega q^j;q)_\infty} \right\} \quad (\mbox{by } (\ref{aberll}))
\cr
&\qquad \qquad \qquad=  \sum_{n=0}^\infty\sum_{j=0}^\infty\frac{q^{-jn}\;t^{n+j}}{(q;q)_j(q;q)_{n}} \frac{(q;q)_k}{(q;q)_{k-j}} \omega^{k-j} (xq^j)^n \frac{(y/x;q)_n(y\omega q^j;q)_\infty}{(y\omega q^j;q)_n(x\omega q^j;q)_\infty}
\cr
& \qquad \qquad\qquad= \frac{ (y\omega;q)_\infty}{(x\omega;q)_\infty}  \omega^k \sum_{j=0}^\infty  \frac{ (x\omega;q)_j\, (t/\omega)^j }{(y\omega ,q;q)_j}   \frac{(q;q)_k}{(q;q)_{k-j}} \sum_{n=0}^\infty\frac{(y/x;q)_n(xt)^n }{ (y\omega q^j,q;q)_n }
\cr
&\qquad \qquad \qquad= \frac{ (y\omega;q)_\infty}{(x\omega;q)_\infty}  \omega^k \sum_{j=0}^k  \frac{(-1)^jq^{kj-({}^j_2)} (q^{-k},x\omega;q)_j\, (t/\omega)^j }{(y\omega ,q;q)_j}    {}_{2}\Phi_1\left[\begin{array}{r}y/x,0;
 \\\\
y\omega q^j;
 \end{array}
q; xt\right],
\end{align*}
which is the right-hand side of (\ref{dte}).
\end{proof}
\par
 The generalized Rogers-Szeg\"o polynomials \cite{JC1982,HLS-MAA2014} are defined as
\begin{align}
\label{RS-P}
   r_n(x,y)=\sum_{k=0}^n \begin{bmatrix}n \\k \\ \end{bmatrix}_qx^k y^{n-k},
\end{align}
where  \cite{HLS-MAA2014}
\be
r_n(x,y)= T(xD_y)\{y^n\}.\label{RSpo}
\ee

Now, we are in position to give and prove the following Rogers formula for generalized $q$-polynomials by using the $q$-difference equations\textcolor{red}{.}
\begin{thm}\label{thm3d}
Let $ \tilde{L}_n^{(\tilde{r},\tilde{s})} (\alpha,x,y,z,a)$ be the polynomials defined as in (\ref{dedref03}). 
\begin{itemize}
\item   For $|\omega x|<1$, we have:
\begin{multline}\label{dthm3_1}
\sum_{n=0}^\infty\sum_{m=0}^\infty \tilde{L}_{n+m}^{(\tilde{r},\tilde{s})} (\alpha,x,y,z,a)\frac{t^n}{(q;q)_n}\frac{\omega^m}{(q;q)_m}\\= \frac{(y\omega;q)_\infty}{(x\omega;q)_\infty}     \sum_{k=0}^\infty \sum_{j=0}^k q^{k\alpha+\tilde{r}({}^k_2)-\tilde{s}({}^{k+1}_{\;\,\;2})}\frac{(-q^{-\alpha},a;q)_k (\omega z)^k }{(-q;q)_k(q;q)_{k-j}}\frac{(x\omega;q)_j (t/\omega)^j}{(y\omega,q;q)_j}  {}_{2}\Phi_1\left[\begin{array}{r}y/x,0;
 \\\\
y\omega q^j;
 \end{array}
q; xt \right].
\end{multline}
\item   For $\max\{|\omega x|,|xt| \}<1$, and  $\tilde{r}=\tilde{s}=0,$ we have:
\begin{multline}\label{dthm3_2}
\sum_{n=0}^\infty\sum_{m=0}^\infty \tilde{L}_{n+m}^{(0,0)} (\alpha,x,y,z,a)\frac{t^n}{(q;q)_n} \frac{\omega^m}{(q;q)_m}\\= \frac{(y\omega;q)_\infty}{(x\omega;q)_\infty}     \sum_{k=0}^\infty \sum_{j=0}^k q^{k\alpha}\frac{(-q^{-\alpha},a;q)_k (\omega z)^k }{(-q;q)_k(q;q)_{k-j}}\frac{(x\omega;q)_j (t/\omega)^j}{(y\omega,q;q)_j}  {}_{2}\Phi_1\left[\begin{array}{r}y/x,0;
 \\\\
y\omega q^j;
 \end{array}
q; xt \right].
\end{multline}
\end{itemize}
\end{thm}
\begin{proof}[Proof of Theorem \ref{thm3d}]
The right-hand side of equation \eqref{dthm3_1} can equivalently be written  as:
\begin{align*}
f(\alpha,x,y,a,z,\tilde{r},\tilde{s})&=
\frac{ (y\omega;q)_\infty}{(x\omega;q)_\infty}  \sum_{k=0}^\infty \sum_{j=0}^\infty  q^{k\alpha+\tilde{r}({}^k_2)-\tilde{s}({}^{k+1}_{\;\,\;2})}\frac{(-q^{-\alpha},a;q)_k(\omega z)^k  }{(-q;q)_k(q;q)_{k-j} }  \frac{ (x\omega;q)_j\, (t/\omega)^j }{(y\omega ,q;q)_j}   {}_{2}\Phi_1\left[\begin{array}{r}y/x,0;
 \\\\
y\omega q^j;
 \end{array}
q; xt\right]\cr
&=    \sum_{k=0}^\infty q^{k\alpha+\tilde{r}({}^k_2)-\tilde{s}({}^{k+1}_{\;\,\;2})}\frac{(-q^{-\alpha},a;q)_k z^k }{(q^2;q^2)_k } \sum_{j=0}^k \frac{(y\omega q^j;q)_\infty (q;q)_k \omega^{k-j}  t^j}{ (x\omega q^j;q)_\infty(q;q)_j (q;q)_{k-j}} {}_{2}\Phi_1\left[\begin{array}{r}y/x,0;
 \\\\
y\omega q^j;
 \end{array}
q; xt \right].\cr
&=   \sum_{k=0}^\infty  \frac{(-q^{-\alpha},a;q)_k }{(q^2;q^2)_k } q^{k\alpha+\tilde{r}({}^k_2)-\tilde{s}({}^{k+1}_{\;\,\;2})}z^kT(tD_\omega)\left\{\frac{(y\omega;q)_\infty}{(x\omega;q)_\infty}  \omega^k\right\}
\cr
&= T(tD_\omega)\left\{\frac{(y\omega;q)_\infty}{(x\omega;q)_\infty} \sum_{k=0}^\infty  \frac{(-q^{-\alpha},a;q)_k }{(q^2;q^2)_k } q^{k\alpha+\tilde{r}({}^k_2)-\tilde{s}({}^{k+1}_{\;\,\;2})} (\omega z)^k\right\}\;\, (\mbox{by using } (\ref{thmc3_1}))
\cr
&= T(tD_\omega)\left\{\sum_{m=0}^\infty \tilde{L}_{m}^{(\tilde{r},\tilde{s})} (\alpha,x,y,z,a)\frac{\omega^m}{(q;q)_m}\right\}
\cr
&= \sum_{m=0}^\infty \tilde{L}_{m}^{(\tilde{r},\tilde{s})} (\alpha,x,y,z,a)\frac{1}{(q;q)_m}T(tD_\omega)\{\omega^m\}\quad (\mbox{by } (\ref{RSpo}))
\cr
&= \sum_{m=0}^\infty \tilde{L}_{m}^{(\tilde{r},\tilde{s})} (\alpha,x,y,z,a)\frac{ r_m(t,\omega) }{(q;q)_m}\cr
&=\sum_{m=0}^\infty \tilde{L}_{m}^{(\tilde{r},\tilde{s})} (\alpha,x,y,z,a)\frac{1}{(q;q)_m} \sum_{n=0}^m\begin{bmatrix}
                        m \\
                       n \\
    \end{bmatrix}_q t^n\omega^{m-n}\cr
&=\sum_{n=0}^\infty\sum_{m=n}^\infty \tilde{L}_{m}^{(\tilde{r},\tilde{s})} (\alpha,x,y,z,a)\frac{t^n}{(q;q)_n}\frac{\omega^{m-n}}{(q;q)_{m-n}}.
\end{align*}
By  setting $m  \to m+n$, 
 we get the left-hand side of (\ref{dthm3_1}). This completes the proof.
\end{proof}

As a special case of Theorem \ref{thm3d}, if   we take $\alpha \to \infty,\,\tilde{r}=\tilde{s}=0,\,x\to y,\,y\to x,\,z=-z$ and $a=-q$, in (\ref{dthm3_1}), we obtain the following corollary:
\begin{corollary}[{\cite[Theorem 3.1]{MAA2016}}]
  For $\max\{|yt|,|\omega y|\}<1$, we have:
\begin{align}\label{cr}
&\sum_{n=0}^\infty\sum_{m=0}^\infty  F_{n+m} (x,y,z;q) (-1)^{n+m} q^{({}^{n+m}_{\,\;\;\;2})}\frac{t^n}{(q;q)_n}\frac{\omega^m}{(q;q)_m}\cr
&\qquad\quad\quad\qquad\qquad = \frac{(x\omega, z\omega;q)_\infty}{(y\omega;q)_\infty}     \sum_{j=0}^\infty  \frac{(-1)^jq^{ ({}^{j}_2)} (y\omega;q)_j (zt)^j}{(x\omega, z\omega,q;q)_j}  {}_{2}\Phi_1\left[\begin{array}{r}x/y,0;
 \\\\
x\omega q^j;
 \end{array}
q; yt \right].
\end{align}
\end{corollary}
\section{Mixed generating function for generalized $q$-polynomials}
\label{section5}

The Hahn polynomials \cite{Hahn049,Hahn49}
 (or Al-Salam and Carlitz polynomials \cite{AlSalam,Cao2012A}) are defined as
\be
\label{ALSALAM}
\phi_n^{(\sigma)}(x|q)=\sum_{k=0}^n
\begin{bmatrix}
                        n \\
                      k \\
    \end{bmatrix}_q
 (\sigma;q)_k\, x^k.
\ee
While Cao   \cite{Cao2012A} used the technique of exponential operator decomposition, Srivastava and Agarwal \cite{SrivastavaAgarwal} adopted the method of transformation theory to deduce the following results  (For more information, please refer to \cite{Hahn049,Hahn49,AlSalam,SrivastavaAgarwal,Cao2009A,Cao2010A,Cao2012A}:
  \begin{lemma}[{\cite[Eq. (3.20)]{SrivastavaAgarwal}}]
  For $\max\{|t|, |xt|\}<1$, we have: 
\label{LEMMA41}
\be
\label{21sums}
\sum_{n=0}^\infty \phi_n^{(\sigma)}(x|q) (\lambda;q)_n\frac{  t^n}{(q;q)_n} \\
= \frac{(\lambda t; q)_\infty }{(t;q)_\infty}   {}_2\Phi_1\left[
\begin{array}{rr} \lambda, \sigma;\\\\
 \lambda  t; \end{array}\,q; xt
\right].
\ee

 \end{lemma}
In Theorem \ref{TA1} below, we give  and prove a mixed generating function for generalized $q$-polynomials
by applying the $q$-difference equations.
  \begin{thm}
\label{TA1}
  Let $\tilde{r},\,\tilde{s} \in\mathbb{Z}$. For $|ut|<1,$  we have: 
\begin{align}\label{prof}
 &\sum_{n=0}^\infty \phi_n^{(\sigma)}(x|q) \tilde{L}_n^{(\tilde{r},\tilde{s})} (\alpha,u,v,z,a)  \frac{ \,t^n}{(q;q)_n}\cr
  &\qquad\qquad =  \frac{(v t; q)_\infty   }{ (ut;q)_\infty} \sum_{m=0}^\infty \sum_{k=0}^\infty\sum_{j=0}^m\frac{(\sigma;q)_m\,x^m}{(q;q)_m}\frac{(q^{-m},ut;q)_j q^j}{(vt,q;q)_j}
  \frac{(-q^{-\alpha};q)_k(a;q)_k\,(tzq^j)^k}{(q^2;q^2)_k }
q^{k\alpha+\tilde{r}({}^k_2)-\tilde{s}({}^{k+1}_{\;\,\;2})}.
\end{align}
\end{thm}
 In the proof of Theorem \ref{TA1},  the following $q$-Chu-Vandermonde  formula will be needed:
 \begin{lemma}[{$q$-Chu-Vandermonde   \cite[  Eq. (II.6)]{GasparRahman}}]
 \be
 \label{male}
   {}_2\Phi_1\left[
\begin{array}{rr}q^{-n},a;\\\\
 c; \end{array}\,q;  q
\right] =\frac{(c/a;q)_n}{(c;q)_n}a^n.
\ee
 \end{lemma}
  \begin{proof}[Proof of   Theorem \ref{TA1}]
Eq. (\ref{prof}) can equivalently be written  as follows:
\begin{align}\label{rhs}
 &  \sum_{n=0}^\infty \phi_n^{(\sigma)}(x|q) \tilde{L}_n^{(\tilde{r},\tilde{s})} (\alpha,u,v,z,a)  \frac{ \,t^n}{(q;q)_n} \cr
& \qquad\qquad = \sum_{m=0}^\infty\frac{(\sigma;q)_m\,x^m}{(q;q)_m} \sum_{j=0}^m\frac{(q^{-m};q)_j q^j}{(q;q)_j} \sum_{k=0}^\infty q^{k\alpha+\tilde{r}({}^k_2)-\tilde{s}({}^{k+1}_{\;\,\;2})}\frac{(-q^{-\alpha};q)_k(a;q)_k\,z^k}{(q^2;q^2)_k }D_{uv}^k\left\{ \frac{(v tq^j; q)_\infty   }{ (ut q^j;q)_\infty}   \right\}.
\end{align}
If we use $g(\alpha,u,v,a,z,\tilde{r},\tilde{s})$ to denote the right-hand side of (\ref{rhs}),
it is easy to see that $g(\alpha,u,v,a,z,\tilde{r},\tilde{s})$ satisfies (\ref{thm1_1}).
Letting $\displaystyle
g(\alpha,u,v,a,z,\tilde{r},\tilde{s})=\sum_{k=0}^\infty B_k(\alpha,u,v,a,\tilde{r},\tilde{s})z^k$ and
\begin{align}
B_k(\alpha,u,v,a,\tilde{r},\tilde{s})= q^{k\alpha+\tilde{r}({}^k_2)-\tilde{s}({}^{k+1}_{\;\,\;2})}\frac{(-q^{-\alpha};q)_k(a;q)_k }{(q^2;q^2)_k }  D_{uv}^k\left\{ \sum_{m=0}^\infty\frac{(\sigma;q)_m\,x^m}{(q;q)_m} \sum_{j=0}^m\frac{(q^{-m};q)_j q^j}{(q;q)_j} \frac{(v tq^j; q)_\infty   }{ (ut q^j;q)_\infty}   \right\}, \label{qqt}
\end{align}
we obtain
\begin{align}
 B_0(\alpha,u,v,a,\tilde{r},\tilde{s})
&= \frac{(v t; q)_\infty   }{ (ut ;q)_\infty}    \sum_{m=0}^\infty\frac{(\sigma;q)_m\,x^m}{(q;q)_m} \sum_{j=0}^m\frac{(q^{-m},ut;q)_j q^j}{(vt,q;q)_j}\cr
&= \frac{(v t; q)_\infty   }{ (ut ;q)_\infty}    \sum_{m=0}^\infty\frac{(\sigma;q)_m\,x^m}{(q;q)_m} {}_2\Phi_1\left[
\begin{array}{rr}q^{-m},ut;\\\\
 vt; \end{array}\,q; q
\right]\quad (\mbox{by } (\ref{male}))\cr
&= \frac{(v t; q)_\infty   }{ (ut ;q)_\infty}    \sum_{m=0}^\infty\frac{(\sigma;q)_m\,x^m}{(q;q)_m}.\frac{(v/u;q)_m (ut)^m}{(vt;q)_m}\cr
&=\frac{(v t; q)_\infty }{(ut;q)_\infty}   {}_2\Phi_1\left[
\begin{array}{rr} v/u,  \sigma;\\\\
 v  t; \end{array}\,q; uxt
\right]\cr
&=\sum_{n=0}^\infty   \phi_n^{(\sigma)}(x|q) \frac{ p_{n }(u,v)\, t^n }{(q;q)_{n}},
 \label{dqt}
\end{align}
and $g(\alpha,u,v,a,0,\tilde{r},\tilde{s})= B_0(\alpha,u,v,a,\tilde{r},\tilde{s})$. Taking equation (\ref{dqt}) into (\ref{qqt}), we get:
 \begin{align}
B_k(\alpha,u,v,a,\tilde{r},\tilde{s})=q^{k\alpha+\tilde{r}({}^k_2)-\tilde{s}({}^{k+1}_{\;\,\;2})}\frac{(-q^{-\alpha},a;q)_k}{(q^2;q^2)_k }\cdot D_{uv}^k\{ B_0(\alpha,u,v,a,\tilde{r},\tilde{s})\}.
\end{align}
Through above identities, $g(\alpha,u,v,a,z,\tilde{r},\tilde{s})$ satisfies the equation (\ref{thm1_1}). So, we have
\begin{align*}
g(\alpha,u,v,a,z,\tilde{r},\tilde{s})&=
  \sum_{k=0}^n\frac{(-q^{-\alpha},a;q)_k }{(q^2;q^2)_k } q^{k\alpha+\tilde{r}({}^k_2)-\tilde{s}({}^{k+1}_{\;\,\;2})} z^k D_{uv}^k\left\{\sum_{n=0}^\infty   \phi_n^{(\sigma)}(x|q) \frac{ p_{n }(u,v)\, t^n }{(q;q)_{n}}\right\}
\cr&=  \sum_{n=0}^\infty  \sum_{k=0}^n\frac{(-q^{-\alpha},a;q)_k }{(q^2;q^2)_k } q^{k\alpha+\tilde{r}({}^k_2)-\tilde{s}({}^{k+1}_{\;\,\;2})} z^k  \phi_n^{(\sigma)}(x|q)\frac{p_{n-k}(u,v)\;t^n}{(q;q)_{n-k}}\cr
&=\sum_{n=0}^\infty  \phi_n^{(\sigma)}(x|q)\frac{t^n}{(q;q)_n}\sum_{k=0}^n\begin{bmatrix}
                        n \\
                       k \\
    \end{bmatrix}_q\begin{bmatrix}
                        \alpha \\
                       k \\
    \end{bmatrix}_{-q}q^{\tau(\tilde{r},\tilde{s})+({}^k_2)} (a;q)_kp_{n-k}(u,v)z^k\cr
&=\sum_{n=0}^\infty  \phi_n^{(\sigma)}(x|q) \tilde{L}_n^{(\tilde{r},\tilde{s})} (\alpha,u,v,z,a)\frac{t^n}{(q;q)_n},
\end{align*}
which  is the left hand side of (\ref{prof}).  The proof is complete.
 \end{proof}
As a special case of Theorem \ref{TA1}, if   we take $\alpha\to \infty,\;a=-q,\,u\to v,\,v\to u,\, z=-z$ and $\tilde{r}=\tilde{s}=0$, we have the following corollary:
  \begin{corollary}[Mixed generating function for trivariate $q$-polynomials $F_n(x,y,z;q)$]
\label{Tc2A1}
 For $\max\{|vt|,|x|\}<1,$ we have:
\begin{align}\label{c2prof}
  \sum_{n=0}^\infty \phi_n^{(\sigma)}(x|q) F_n(u,v,z;q)  \frac{ (-1)^n q^{({}^n_2)}\,t^n}{(q;q)_n}  =  \frac{(\sigma x,ut,zt;q)_\infty  }{ (vt,x;q)_\infty}  {}_4\Phi_3\left[
\begin{array}{rr} \sigma, vt, 0,0;\\\\
 ut,zt, q/x; \end{array}\,q; q
\right].
\end{align}
\end{corollary}

\section{Concluding remarks}
\label{conclusion}

In  this paper, by using the method of $q$-difference equations, we have systematically deduced several types of generating functions for certain $q$-polynomials. In fact, $q$-Laguerre polynomials and Cigler's polynomials have double $q$-binomial coefficients. Computing their generating functions reveals to be  too difficult. We have noticed that the $q$-difference equations are important tools to calculate  {generating} functions for $q$-polynomials. It is necessary to construct $q$-difference equations satisfied by the general $q$-polynomials with double $q$-binomial coefficients. We have then focused on the expansion of a function of many variables and on some $q$-polynomials. Therefore, we have searched for the generalized $q$-difference equations for general $q$-polynomials with double $q$-binomial coefficients. We believe that this work will be a motivation  to study other $q$-polynomials and their applications.
%
%
\section*{Acknowledgments}
The authors are grateful to the referees and editor for their useful comments and suggestions to improve the paper. 
This work was supported by the Zhejiang Provincial Natural Science Foundation of China (No. LY21A010019). The ICMPA-UNESCO Chair is in partnership
with the Association pour la Promotion Scientifique de l'Afrique
(APSA), France, and Daniel Iagolnitzer Foundation (DIF), France,
supporting the development of mathematical physics in Africa.

\end{CJK*}
\end{document}